\documentclass[12pt,letter]{amsart}
\usepackage{amsfonts}
\usepackage{amsthm}
\usepackage{amsmath}
\usepackage{amssymb} %
\usepackage{slashed}
\usepackage{amscd}
\usepackage[latin2]{inputenc}
\usepackage{t1enc}
\usepackage[mathscr]{eucal}
\usepackage{indentfirst}
\usepackage{graphicx}
\usepackage{graphics}
\usepackage{pict2e}
\usepackage{epic}
\numberwithin{equation}{section}
\usepackage[margin=2.9cm]{geometry}
\usepackage{epstopdf} 

\usepackage{tikz}
\usepackage{dsfont}
\usetikzlibrary{matrix}

\newcommand\CP{{\mathbb C}\mathbb{P}} 
\newcommand\QQ{{\mathbb Q}} 
\newcommand\RR{{\mathbb R}} 
\newcommand\ZZ{{\mathbb Z}} 
 
\theoremstyle{plain}
\newtheorem{theorem}{Theorem}[section]
\newtheorem*{theorem*}{Theorem}
\newtheorem{lemma}[theorem]{Lemma}

\newtheorem{corollary}[theorem]{Corollary}
\newtheorem{proposition}[theorem]{Proposition}

 \theoremstyle{definition}

\newtheorem{remark}[theorem]{Remark}
\newtheorem{?}[theorem]{Problem}

\begin{document}

\title[On the minimal sum of Betti numbers of an almost complex manifold]{On the minimal sum of Betti numbers of an almost complex manifold}

\author{Michael Albanese and Aleksandar Milivojevi\'c}

\begin{abstract} We show that the only rational homology spheres which can admit almost complex structures occur in dimensions two and six. Moreover, we provide infinitely many examples of six-dimensional rational homology spheres which admit almost complex structures, and infinitely many which do not. We then show that if a closed almost complex manifold has sum of Betti numbers three, then its dimension must be a power of two. \end{abstract}

\address{Stony Brook University \\ Department of Mathematics} 

\email{michael.albanese@stonybrook.edu}
\email{aleksandar.milivojevic@stonybrook.edu}

\maketitle

\section{Introduction}

An outstanding problem in the topology of closed smooth manifolds is to determine whether the existence of an integrable complex structure imposes restrictions on the topology of the manifold beyond those imposed by the existence of an almost complex structure itself. In the symplectic setting, the existence of a \textit{closed} non-degenerate two-form $\omega$ on an even-dimensional manifold $M^{2n}$ tells us that the sum of the Betti numbers $\dim H^i(M; \RR)$ is at least $n+1$ (provided by the cohomology classes $[\omega^i]$). Among $2n$--manifolds admitting an almost complex structure $J$ we can consider the possible values of sums of Betti numbers, and one can ask whether requiring the presence of an integrable $J$ would increase these possible values. We show that for manifolds of dimension $2n \geq 8$ not equal to a power of two, the minimal possible sum of Betti numbers is four, both for almost complex and complex manifolds (and this bound is achieved by Hopf and Calabi--Eckmann manifolds, which have the homotopy type of a product of two odd-dimensional spheres).  This follows from our two main results featured in sections 2 and 3 respectively.

\begin{theorem*} Let $M$ be a rational homology sphere. If $M$ admits an almost complex structure, then $\dim M = 2$ or $6$. \end{theorem*}

\begin{theorem*}
Let $M$ be a closed almost complex manifold with sum of Betti numbers three. Then $\dim M$ is a power of two.
\end{theorem*}

In real dimension 6 there is a possible discrepancy between the minimal sum of Betti numbers for almost complex and complex manifolds caused by rational homology six-spheres.


We summarize the situation in the graph below, where the horizontal axis denotes real dimension $2n$, and the vertical axis denotes the minimal sum of Betti numbers among almost complex manifolds in the given dimension. The empty circle at $(6,4)$ denotes the minimal sum of Betti numbers among known complex threefolds (achieved by $\CP^3$). By direct calculation on small powers of two, it is observed that the smallest dimension greater than $4$ where the minimal sum of Betti numbers could be three is 2048. 

We communicate a conjecture of Sullivan that the minimal sum of Betti numbers of a compact complex $n$-fold, $n\geq 3$, is four (which would give the above graph a particularly nice form). This would imply that $S^6$ (or any rational homology six-sphere) does not admit a complex structure; this is an open problem, see [1, Problem 3].

\vspace{0.3em}

The authors would like to thank Dennis Sullivan for bringing this problem to their attention, as well as Blaine Lawson and Claude LeBrun for many helpful conversations.  

\begin{figure}
\begin{tikzpicture}
  \matrix (m) [matrix of math nodes,
    nodes in empty cells,nodes={minimum width=2.5ex,
    minimum height=2.5ex,outer sep=-5pt},
    column sep=0.5ex,row sep=0.25ex]{
         6     &\cdot     &\cdot     &\cdot    &\cdot &\cdot &\cdot &\cdot &\cdot &\cdot &\cdot  &\cdot &\cdot &\cdot &\cdot &\cdot  &\cdot &\cdot &\cdot &\cdot \\
    		 5     &\cdot     &\cdot     &\cdot    &\cdot  &\cdot &\cdot &\cdot &\cdot &\cdot &\cdot  &\cdot &\cdot &\cdot &\cdot &\cdot  &\cdot &\cdot &\cdot &\cdot\\
          4     &\cdot      &\cdot     &\cdot     &\circ  &\bullet  &\bullet  &\bullet  &\bullet  &\bullet  &\bullet  &\bullet  &\bullet  &\bullet  &\bullet  &\bullet    &\bullet  &\bullet  &\bullet  &\bullet  \\
          3     &\cdot     &\cdot     &\bullet     &\cdot  &\cdot &\cdot &\cdot  &\cdot &\cdot &\cdot  &\cdot &\cdot &\cdot &\cdot &\cdot  &\cdot &\cdot &\cdot &\cdot  \\
          2     &\cdot      &\bullet     &\cdot     &\bullet  &\cdot &\cdot &\cdot  &\cdot &\cdot &\cdot  &\cdot &\cdot &\cdot &\cdot &\cdot  &\cdot &\cdot &\cdot &\cdot \\
          1     &\bullet      &\cdot     &\cdot     &\cdot  &\cdot &\cdot &\cdot  &\cdot &\cdot &\cdot  &\cdot &\cdot &\cdot &\cdot &\cdot  &\cdot &\cdot &\cdot &\cdot \\
          0     &\cdot      &\cdot     &\cdot     &\cdot  &\cdot &\cdot &\cdot  &\cdot &\cdot &\cdot  &\cdot &\cdot &\cdot &\cdot &\cdot  &\cdot &\cdot &\cdot &\cdot  \\
    \quad\strut &   0  &  2  &  4  & 6 & 8 & 10 & 12 & 14 & 16 & 18 & 20 & 22 & 24 & 26 & 28 & 30 & 32 & 34 & 36   \strut \\};
\end{tikzpicture}
\caption{The minimal sum of Betti numbers among closed smooth (almost) complex manifolds of a given real dimension.}
\end{figure}
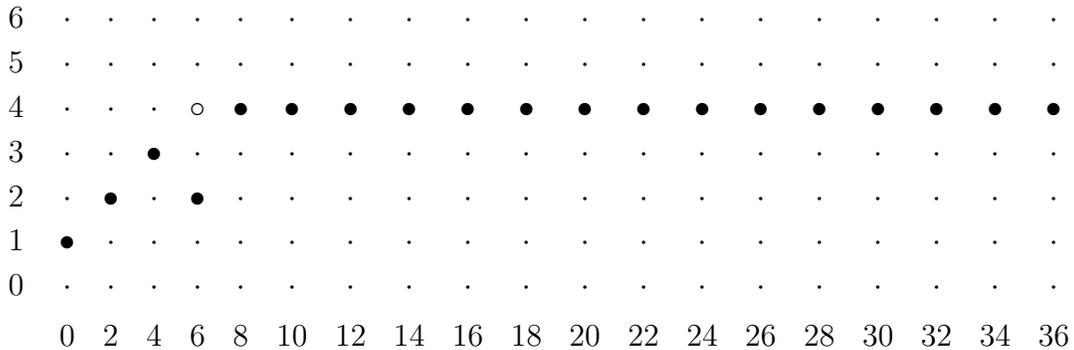

\section{Almost complex structures on rational homology spheres}

Let $R$ be a ring. An $n$-dimensional $R$ \emph{homology sphere} is a closed smooth $n$-dimensional manifold $M$ with $H^*(M; R) \cong H^*(S^n; R)$; when $R = \mathbb{Z}$ we say that $M$ is an \emph{integral homology sphere} and when $R = \mathbb{Q}$ we say that $M$ is a \emph{rational homology sphere}. Note that rational homology spheres (and hence integral homology spheres) are necessarily orientable as $H^n(M; \mathbb{Q}) \neq 0$. Rational homology spheres can alternatively be characterized among closed orientable manifolds as those which have the smallest possible sum of Betti numbers, namely two.

\vspace{0.3em}

Recall the following theorem of Borel and Serre \cite{BS}.

\begin{theorem}
The only spheres which admit almost complex structures are $S^2$ and $S^6$.
\end{theorem}

The modern proof of this fact can be found in many places; for example \cite{KP} is a nice self-contained exposition. There are three steps in the proof. First, one notices that the Chern character $\operatorname{ch} : K(S^{2n}) \to H^*(S^{2n}; \mathbb{Q})$ takes the form $\operatorname{ch}(E) = n + \frac{(-1)^{n+1}}{(n-1)!}c_n(E)$. Second, it follows from Bott Periodicity for $K$-theory that the image of the Chern character is contained in $H^*(S^{2n}; \mathbb{Z})$. Finally, equipping $S^{2n}$ with an almost complex structure is equivalent to realizing $TS^{2n}$ as a rank $n$ complex vector bundle, and so we see that $\operatorname{ch}_n(TS^{2n}) = \frac{(-1)^{n+1}}{(n-1)!}c_n(TS^{2n}) \in H^{2n}(S^{2n}; \mathbb{Z})$, but $c_n(TS^{2n}) = e(TS^{2n})$ which is twice the oriented generator, so $(n - 1)! \mid 2$. This still leaves the possibility that $n = 2$, but this can be ruled out by a direct characteristic class argument.

It should be noted that the original proof by Borel and Serre uses different techniques and actually proves something stronger: if $M$ is a closed $2n$-dimensional almost complex manifold with $c_i(M) = 0$ for $1 \leq i \leq n - 1$, and $p$ is a prime such that $p < n$ and $p \nmid n$, then $p \mid \chi(M)$. It follows that if a $2n$-dimensional integral homology sphere admits an almost complex structure, then $2n \leq 6$. Note, however, that this theorem cannot be used to tackle rational homology spheres. Also see \cite{L} for a similar result.

Here is an alternative to the integrality step in the modern proof that the first author learned from Blaine Lawson. Recall that on a spin manifold $M$ there is a Dirac operator $\slashed{\partial}$ and if $E \to M$ is a complex vector bundle, there is a twisted Dirac operator $\slashed{\partial}_E$ which has index $$\operatorname{ind}(\slashed{\partial}_E) = \int_M \operatorname{ch}(E)\hat{A}(TM),$$ see [10, Theorem III.13.10]. If $M = S^{2n}$, then $\hat{A}(TS^{2n}) = 1$ as $TS^{2n}$ is stably trivial. Taking $E = TS^{2n}$, we see that $\int_{S^{2n}} \operatorname{ch}(TS^{2n}) = \operatorname{ind}(\slashed{\partial}_{TS^{2n}}) \in \mathbb{Z}$.

Using a generalization of this argument, we obtain the following result.

\begin{theorem}
Let $M$ be a rational homology sphere. If $M$ admits an almost complex structure, then $\dim M = 2$ or $6$.
\end{theorem}
\begin{proof}
Suppose $\dim M = 2n$. If $n$ is even, say $n = 2k$, then modulo torsion $p_k(TM) = 2(-1)^kc_{2k}(TM) = 2(-1)^kc_n(TM)$, and all other Pontryagin classes are torsion. By the Hirzebruch signature theorem, \begin{align*} 0 &= \sigma(M) = \int_M L(p_1, \dots, p_k) = \int_M h_kp_k = 2(-1)^kh_k\int_M c_n(TM)\\ & = 2(-1)^kh_k\int_M e(TM) = 2(-1)^kh_k\chi(M) = 4(-1)^kh_k. \end{align*} This is a contradiction as $h_k \neq 0$ (see [3, Corollary 3]).

Now suppose $n$ is odd. The case $n = 1$ is clear, so suppose $n > 1$. An almost complex manifold has a canonical spin${}^c$ structure, so for any complex vector bundle $E \to M$ there is a twisted spin${}^c$ Dirac operator $\slashed{\partial}^c_E$ which has index $$\operatorname{ind}(\slashed{\partial}^c_E) = \int_M \exp(c_1(L)/2)\operatorname{ch}(E)\hat{A}(TM),$$ where $L$ is the complex line bundle associated to the spin${}^c$ structure [5, Theorem 26.1.1]. Note that $c_1(L)$ and $p_i(TM)$ are all torsion classes, so $$\operatorname{ind}(\slashed{\partial}^c_E) = \int_M \operatorname{ch}(E) = \int_M \operatorname{ch}_n(E) = \frac{(-1)^{n+1}}{(n-1)!}\int_M c_n(E).$$ Taking $E = TM$, we see that $\int_M c_n(TM) = \int_M e(TM) = \chi(M) = 2$, and so $(n - 1)! \mid 2$. As $n$ is odd and greater than one, the only possibility is $n = 3$.
\end{proof}

Note that for the canonical spin${}^c$ structure associated to an almost complex structure, the index of $\slashed{\partial}_E^c$ can also be expressed as

$$\operatorname{ind}(\slashed{\partial}_E^c) = \int_M\operatorname{ch}(E)\operatorname{Td}(M).$$

If $J$ is integrable and $E$ is holomorphic, then $\slashed{\partial}_E^c = \bar{\partial}_E + \bar{\partial}_E^*$ and the above becomes the statement of the Hirzebruch-Riemann-Roch Theorem.

There is also a proof for the $n$ even case along the lines of the $n$ odd case. First note that as $\sigma(M) = 0$ and $h_k \neq 0$, we must have $p_k(TM) = 0$ and hence $$\operatorname{ind}(\slashed{\partial}^c_{E}) = \int_M \exp(c_1(L)/2)\operatorname{ch}(E)\hat{A}(TM) = \frac{(-1)^{n+1}}{(n-1)!}\int_M c_n(E)$$ as in the $n$ odd case. The benefit of this alternative proof is that it allows us to deduce the following.


\begin{proposition}
Let $M$ be a $2n$-dimensional rational homology sphere. For any complex vector bundle $E \to M$, we have $(n - 1)! \mid c_n(E)$. Moreover, every class in $H^{2n}(M;\mathbb{Z})$ which is a multiple of $(n-1)!$ is $c_n(E)$ for some $E$.
\end{proposition}

For the last statement, take a degree $d$ map $f : M \to S^{2n}$. Let $E_0 \to S^{2n}$ be a complex vector bundle with $c_n(E_0)$ equal to $(n-1)!$ times the generator of $H^{2n}(S^{2n};\ZZ)$ (for example, the positive spinor bundle $\mathbb{S}^+ \to S^{2n}$). Then $E = f^*E_0$ has $c_n(E)$ equal to $(n-1)!\cdot d$ times the generator of $H^{2n}(M;\ZZ)$. Note that in the case of $S^{2n}$, rank $n$ complex vector bundles are determined up to isomorphism by their top Chern class; it is not clear whether the same is true for rational homology spheres.

\vspace{0.3em}

There is only one rational homology two-sphere up to diffeomorphism, namely $S^2$. However there are many rational homology six-spheres. The natural question which arises is: which of these admit almost complex structures? The primary obstruction to a closed orientable manifold $M$ admitting an almost complex structure is the third integral Stiefel-Whitney class $W_3(M) \in H^3(M;\mathbb{Z})$; if $\dim M = 6$, this is the only obstruction [12, Remark 1]. Recall that $W_3(M) = \beta(w_2(M))$ where $\beta$ is the Bockstein associated to the short exact sequence $0 \to \mathbb{Z} \xrightarrow{\times 2} \mathbb{Z} \to \mathbb{Z}_2 \to 0$. In particular, $W_3(M)$ is two-torsion, i.e. $W_3(M) \in H^3(M;\mathbb{Z})_2 := \{\alpha \in H^3(M; \mathbb{Z}) \mid 2\alpha = 0\}$. So a sufficient condition for a closed orientable six-manifold to admit an almost complex structure is $H^3(M; \mathbb{Z})_2 = 0$. (Note, however, that this is not necessary as $\mathbb{RP}^3\times\mathbb{RP}^3$ demonstrates.) In particular, every integer homology six-sphere admits an almost complex structure. This already provides us with many examples of almost complex rational homology spheres, see \cite{K}.

Note that $W_3(M)$ is the obstruction to finding an integral lift for $w_2(M)$ and hence the obstruction to $M$ being spin${}^c$. Orientable manifolds of dimension at most four are always spin${}^c$, but this is no longer true starting in dimension five. For example, the Wu manifold $SU(3)/SO(3)$ has dimension five and is not spin${}^c$. Moreover, it is a (simply connected) rational homology sphere. It would then seem reasonable to expect that there are six-dimensional rational homology spheres which are not spin${}^c$ and hence do not admit almost complex structures.

Let $M$ be a closed $n$-dimensional manifold and consider the $(n + 1)$-dimensional manifold $N$ obtained as the result of surgery on the $S^1$ factor of $S^1\times M$. (This process is known as \emph{spinning $M$}.) If $M$ is a rational homology sphere, then so is $N$. 

\begin{proposition}
Let $M$ be a closed smooth orientable $n$-dimensional manifold and let $N$ denote the $(n + 1)$-dimensional closed manifold obtained by spinning $M$. Then $N$ is spin${}^c$ if and only if $M$ is.
\end{proposition}
\begin{proof}
Let $U = (S^1\times M)\setminus (S^1\times D^n)$, $V = S^1\times D^n$ and $W = D^2\times S^{n-1}$, and let $i : U \to S^1\times M$, $j : U \to N$, and $k : W \to N$ be the inclusions. 

Suppose $n \geq 5$. By applying Mayer-Vietoris to $(U, V)$ and $(U, W)$ we see that $i^* : H^3(S^1\times M; \mathbb{Z}) \to H^3(U; \mathbb{Z})$ and $j^* : H^3(N; \mathbb{Z}) \to H^3(U; \mathbb{Z})$ are isomorphisms. By naturality, we have $i^*W_3(S^1\times M) = W_3(U) = j^*W_3(N)$. If $p : S^1\times M \to M$ is the projection, then $p^*$ is injective and $W_3(S^1\times M) = p^*W_3(M)$. The claim now follows.

If $n = 4$, the same computation shows that $i^* : H^3(S^1\times M; \mathbb{Z}) \to H^3(U; \mathbb{Z})$ and $(j^*, k^*) : H^3(N; \mathbb{Z}) \to H^3(U; \mathbb{Z})\oplus H^3(W; \mathbb{Z})$ are injective. Note that $W_3(S^1\times M) = p^*W_3(M) = p^*0 = 0$ as $n = 4$, so $W_3(U) = i^*W_3(S^1\times M) = 0$. Now $(j^*, k^*)W_3(N) = (j^*W_3(N), k^*W_3(N)) = (W_3(U), W_3(W)) = (0, 0)$ so $W_3(N) = 0$ by injectivity.

If $n \leq 3$, the claim automatically holds.
\end{proof}

Rational homology spheres in odd dimensions are plentiful, for example lens spaces. Spinning any lens space gives an even-dimensional rational homology sphere; moreover, as lens spaces are all spin${}^c$, the resulting manifold is also spin${}^c$. In particular, spinning five-dimensional lens spaces gives infinitely many examples of six-dimensional rational homology spheres which admit almost complex structures; note, none of these are integral homology spheres. On the other hand, spinning the Wu manifold gives a six-dimensional rational homology sphere which is not spin${}^c$.

\begin{corollary}
Not all six-dimensional rational homology spheres admit almost complex structures.
\end{corollary}

By taking connected sums, we see that there are also infinitely many six-dimensional rational homology spheres which do not admit almost complex structures. Spinning the Wu manifold repeatedly, we obtain the following result which may be of independent interest.

\begin{corollary}
For every $n \geq 5$ there are simply connected $n$-dimensional rational homology spheres which are not spin${}^c$.
\end{corollary}

\section{Sum of betti numbers 3 in dimensions not a power of 2}

We now show that a closed almost complex manifold $M^n$, with $n$ not equal to a power of two, cannot satisfy the property that the sum of its Betti numbers is equal to three. First, observe that by Poincar\'e duality $n$ must be even and the rational cohomology must be concentrated in degrees $0, \tfrac{n}{2}, n$. Namely, $H^0(M; \QQ) \cong H^{n/2}(M;\QQ) \cong H^{n}(M;\QQ) \cong \QQ$ and all the other rational cohomology groups are trivial. Now, if $n$ were of the form $4k+2$, then the intersection pairing on $H^{2k+1}$ would be skew-symmetric and so necessarily of even rank, which is a contradiction. Therefore we can restrict to manifolds of dimension $n = 4k$. We will use the following lemma. 

\begin{lemma}
Let $n$ be a positive integer and suppose $n = 2^ka$ where $a$ is odd. Denote the number of factors of 2 in $n!$ by $l$. Then $l \leq n - 1$ with equality if and only if $a = 1$. 
\end{lemma}
\begin{proof}
Note that $l = \sum_{i=1}^{\infty}\lfloor\frac{n}{2^i}\rfloor = \sum_{i=1}^{k} 2^{k-i}a + \sum_{i=1}^{\infty}\lfloor\frac{a}{2^i}\rfloor = (2^k - 1)a + \sum_{i=1}^{\infty}\lfloor\frac{a-1}{2^i}\rfloor \leq (2^k - 1)a + \sum_{i=1}^{\infty}\frac{a-1}{2^i} = (2^k - 1)a + (a - 1) = n - 1$ with equality if and only if $\lfloor\frac{a-1}{2^i}\rfloor = \frac{a-1}{2^i}$ for every $i$; as the left hand side is zero for $i$ large enough, the same must be true of the right hand side and hence $a = 1$.
\end{proof}


First we consider dimensions of the form $8k$.

\begin{proposition} There is no closed almost complex manifold of dimension $8k$, with $k$ not a power of two, whose sum of Betti numbers equals three.\end{proposition}

\begin{proof} Suppose $M$ is such a manifold. Due to the rational cohomology of $M$ being concentrated in degrees $0,4k,8k$, the only non-trivial Pontryagin classes modulo torsion are $p_0, p_k, p_{2k}$. By Hirzebruch [6, p.777], on an $8k$--dimensional closed almost complex manifold, the signature is equal to the Euler characteristic modulo 4. Since $H^{4k}(M;\QQ)$ is one-dimensional, the signature must be $-1$, and so by the Hirzebruch signature theorem we have $$\int_M h_{2k} p_{2k} + h_{k,k} p_k^2 = - 1.$$ In [5, p.12] we find the following explicit formula for the leading coefficients $h_m$ of the $m$th Hirzebruch polynomial $L_m$, $$h_m = \frac{2^{2m} (2^{2m-1}-1)}{(2m)!}B_m.$$ Here $B_m$ denotes the $m$th non-trivial Bernoulli number without sign. That is, $$B_1 = \tfrac{1}{6}, B_2 = \tfrac{1}{30}, B_3 = \tfrac{1}{42}, \ldots $$ Let us denote by $N_k$ and $D_k$ the numerator and denominator of $B_k$ in maximally reduced form. It is well known by the von Staudt--Clausen theorem that $D_k$ is the product of (distinct) primes $p$ such that $p-1$ divides $2k$. In particular, $D_k$ contains a single factor of $2$. On the other hand, $N_k$ is a product of odd primes. 

As for the coefficient $h_{k,k}$ that appears in the above integral, by [3, Theorem 5] we have the equation $$h_{k,k} = \tfrac{1}{2} h_k^2 - \tfrac{1}{2} h_{2k},$$ where $h_k$ denotes the coefficient of $p_k$ in the $k^{\textrm{th}}$ Hirzebruch polynomial whose formula is given above.

Since $M$ is assumed to admit an almost complex structure, modulo torsion we have the following relations between the Pontryagin and Chern classes, \begin{align*} p_k &= 2 (-1)^k c_{2k}, \\ p_{2k} &= c_{2k}^2 + 2c_{4k}.\end{align*} Again since $H^{4k}(M; \QQ)$ is one-dimensional, there is an integer $\eta$ such that $c_{2k} = \eta a$ (plus torsion) for a fixed generator $a$ of the free part of $H^{4k}(M; \ZZ)$.

Evaluating the signature formula integral, we obtain $$\eta^2 (h_{2k} + 4h_{k,k}) = 6h_{2k} + 1.$$ Using the above expression for $h_{k,k}$, we have $\eta^2(2h_k^2 - h_{2k}) = 6h_{2k} + 1$. Now, by the above formula for $h_k, h_{2k}$, after clearing denominators from here we obtain \begin{align*} \eta^2  2^{4k+1}  (2^{2k-1}-1)^2  (4k)!  N_k^2  D_{2k} - \eta^2  2^{4k}  (2^{4k-1}-1)  (2k)!^2  N_{2k}  D_k^2 \\  = 3  \cdot 2^{4k+1}  (2^{4k-1}-1)  (2k)!^2  N_{2k}  D_k^2 + (2k)!^2  (4k)!  D_k^2  D_{2k}.\end{align*}

We count the number of factors of 2 in the four terms of this equation. Recall that $D_k$ and $D_{2k}$ each contain a single factor of 2. Let us denote the number of factors of $2$ in $(2k)!$ by $l$, and observe that the number of factors of $2$ in $(4k)!$ is given by $ \lfloor \tfrac{4k}{2} \rfloor + \lfloor \tfrac{4k}{4} \rfloor + \cdots = 2k + \lfloor \tfrac{2k}{2} \rfloor + \lfloor \tfrac{2k}{4} \rfloor + \cdots = 2k+l.$ Since $k$ is not a power of two, using Lemma 3.1. we conclude $2k+l \leq 4k-2$. We can also bound the number of factors of 2 in $(4k)!$ from below by $2l+2$. Indeed, since $l\leq 2k-2$ by Lemma 3.1, we have $2k+l \geq 2l+2$. Now we see that the first term on the left hand side in the equation above contains at least $4k+2l+4$ factors of 2. The second term on the left hand side has at least $4k+2l+2$.  The first term on the right hand side contains exactly $4k+2l+3$ factors of 2, while the second one contains at most $4k+2l+1$. This contradiction in divisibility tells us that no such $M^{8k}$ can exist.
\end{proof}

In [13, Lemma 2.3] it was shown that a closed smooth manifold with sum of Betti numbers three can only occur in dimension 4 or in dimensions of the form $8k$. (In this paper, the goal was to find simply connected such manifolds, but the argument does not require any assumptions on the fundamental group.) This lets us exclude the case of dimension $8k+4$ in our consideration of almost complex manifolds with sum of Betti numbers equal to three, as no such smooth manifolds exist to begin with. To summarize, we have the following result. 

\begin{theorem}
Let $M$ be a closed almost complex manifold with sum of Betti numbers three. Then $\dim M$ is a power of two.
\end{theorem}

\section{Remarks on dimensions equal to a power of 2}


By Adams' solution of the Hopf invariant one problem, any $2n$-dimensional manifold admitting a minimal cellular decomposition with three cells (that is, one 0-cell, one $n$-cell, and one $2n$-cell) has the homotopy type of $\mathbb{R}\mathbb{P}^2$, $\CP^2$, $\mathbb{H}\mathbb{P}^2$, or $\mathbb{O}\mathbb{P}^2$. Considering the relaxed constraint of having rational cohomology ring $\QQ[\alpha]/(\alpha^3)$, by [4, Theorem A] we know that any such manifold must have dimension of the form $8(2^a + 2^b)$, though the only known examples are in dimensions that are a power of two. Known examples exist in dimensions beyond 16. Indeed, in [7, Theorem A], examples are given of (simply connected) closed manifolds in dimensions 32, 128, and 256 with rational cohomology ring $\QQ[\alpha]/(\alpha^3)$. 

Suppose $M^{8k}$ is a closed almost complex manifold with sum of Betti numbers three. By Theorem 3.3, we know that $k$ must be a power of two. Consider the Chern classes $c_{2k}$ and $c_{4k}$. We know that $\int_M c_{4k} = 3$ and that the free part of $c_{2k}$ as an integral class is of the form $\eta a$ for some fixed generator $a$ of the free part of $H^{4k}(M; \ZZ)$. We observe some conditions on the prime factors of $\eta$. 

\begin{proposition} The coefficient $\eta$, where $c_{2k} = \eta a$ (plus torsion), is odd. Furthermore, it is not divisible by nine.\end{proposition}

\begin{proof}

Recall the equation $$\eta^2 (h_{2k} + 4h_{k,k}) = 6h_{2k} + 1$$ obtained in the proof of Proposition 3.2 (which is still valid in the case of $8k$ equal to a power of two). It is known that the common denominator of all the terms in any Hirzebruch $L$-polynomial is odd [5, Lemma 1.5.2], so after clearing denominators, this equation becomes $$\eta^2(\alpha + 4\beta) = 6\gamma + \delta,$$ where $\alpha,\beta,\gamma,\delta$ are integers and $\delta$ is odd. Looking at this equation modulo 2, it follows that $\eta$ must be odd.

Now we show that $\eta$ cannot be divisible by 9. Consider the equation \begin{align*} \eta^2  2^{4k+1}  (2^{2k-1}-1)^2  (4k)!  N_k^2  D_{2k} - \eta^2  2^{4k}  (2^{4k-1}-1)  (2k)!^2  N_{2k}  D_k^2 \\  = 3  \cdot 2^{4k+1}  (2^{4k-1}-1)  (2k)!^2  N_{2k}  D_k^2 + (2k)!^2  (4k)!  D_k^2  D_{2k}\end{align*} obtained as in the proof of Proposition 3.2. We count the factors of 3 in each summand. Denote the number of factors of 3 in $\eta$ and $(2k)!$ by $m$ and $l$ respectively. We note that since $\lfloor 2\cdot \rfloor \geq 2\lfloor \cdot \rfloor$, the number of factors of 3 in $(4k)!$ is at least $2l$. By the von Staudt--Clausen theorem, $D_k$ and $D_{2k}$ each contain exactly one factor of 3; the numerators $N_k$ and $N_{2k}$ are not divisible by 3. Also note that $2^{2k-1}-1$ and $2^{4k-1}-1$ are congruent to 1 mod 3. Now we see that the four terms in the equation above contain at least $2m+2l+1$, exactly $2m+2l+2$, exactly $2l+3$, and at least $4l+3$ factors of 3 respectively. We conclude $m \leq 1$.
\end{proof}

\begin{remark} Similarly one shows that $\eta$ is not divisible by 5 if $8k\geq 16$, by 17 if $8k\geq 64$, by 257 if $8k\geq 1024$, or by 65537 if $8k \geq 2^{18}$. On the other hand, as we increase the dimension $8k$, the coefficient $\eta^2$ in any such almost complex manifold must tend to infinity. Namely, by the well-known relation $$B_k = \frac{(2k)!}{2^{2k-1} \cdot \pi^{2k}}\zeta(2k) $$ between the Bernoulli numbers and the Riemann zeta function, we have that $B_k$ tends to $\tfrac{(2k)!}{2^{2k-1} \cdot \pi^{2k}}$ as $k\to \infty$. From here it follows that the coefficient $h_k$ tends to zero as $k\to \infty$. Note that from $\eta^2(2h_k^2 - h_{2k}) = 6h_{2k} + 1$ and the fact that $h_k$ and $h_{2k}$ are positive, we obtain $\eta^2 \geq \tfrac{1}{2 h_k^2}$. 

\vspace{0.5em}

By direct calculation with the equation $\eta^2 (2h_k^2 - h_{2k}) = 6h_{2k}+1$ used above, we see that there are no integer solutions for $\eta$ if $8k \in \{8,16,32,64,128,256,512,1024\}$. The only possible solutions of the equations satisfy $\eta^2 \in \QQ \setminus \ZZ$, except for the case of $k=1$ where the equation is $\eta^2 = 29$. So, there are no almost complex manifolds in dimensions between $8$ and $1024$ whose sum of Betti numbers is three. \end{remark}



\begin{thebibliography}{99} 
\bibitem{} Agricola, I., Bazzoni, G., Goertsches, O., Konstantis, P., and Rollenske, S., 2018. \textit{On the history of the Hopf problem}. Differential Geometry and its Applications, Vol. 57, pp.1-9 
\bibitem{BS} Borel, A., 1953. \textit{Groupes de Lie et puissances r\'eduites de Steenrod}. American Journal of Mathematics, 75(3), pp.409-448.
\bibitem{} Berglund, A. and Bergstr\"om, J., 2017. \textit{Hirzebruch L-polynomials and multiple zeta values}. Mathematische Annalen, pp.1-13. 
\bibitem{} Fowler, J. and Su, Z., 2016. \textit{Smooth manifolds with prescribed rational cohomology ring}. Geometriae Dedicata, 182(1), pp.215-232.
\bibitem{H} Hirzebruch, F., Borel, A. and Schwarzenberger, R.L.E., 1966. \textit{Topological methods in algebraic geometry} (Vol. 175). Berlin-Heidelberg-New York: Springer.
\bibitem{Hirz} Hirzebruch, F., 1987. Gesammelte Abhandlungen: Band I: 1951-1962; Band II: 1963-1987. Springer Verlag.
\bibitem{KS} Kennard, L. and Su, Z., 2017. \textit{On dimensions supporting a rational projective plane}. Journal of Topology and Analysis, pp.1-21. 
\bibitem{K} Kervaire, M.A., 1969. \textit{Smooth homology spheres and their fundamental groups}. Transactions of the American Mathematical Society, 144, pp.67-72.
\bibitem{KP} Konstantis, P. and Parton, M., 2018. \textit{Almost complex structures on spheres}. Differential Geometry and its Applications, Vol. 57, pp.10-22.
\bibitem{LM} Lawson, H.B. and Michelsohn, M.L., 1989. \textit{Spin geometry} (Vol. 38). Princeton university press.
\bibitem{L} Little, R., 1975. \textit{Obstruction formulas and almost-complex manifolds}. Proceedings of the American Mathematical Society, 50.
\bibitem{MasseyACS} Massey, W.S., 1961. \textit{Obstructions to the existence of almost complex structures}. Bulletin of the American Mathematical Society, 67(6), pp.559-564.
\bibitem{S} Su, Z., 2014. \textit{Rational analogs of projective planes}. Algebraic \& Geometric Topology, 14(1), pp.421-438.

\end{thebibliography}
\end{document}